\documentclass[graybox]{svmult}

\usepackage{mathptmx}       
\usepackage{helvet}         
\usepackage{courier}        
\usepackage{type1cm}        
\usepackage{makeidx}         
\usepackage{graphicx}        
\usepackage{multicol}        
\usepackage[bottom]{footmisc}

\usepackage{mathrsfs}       
\usepackage{amssymb}        
\usepackage{amsmath}        
\usepackage{dsfont}         
\usepackage{nicefrac}       
\usepackage{bm}             

\makeindex             

\newcommand{\Var}{\operatorname{Var}}

\renewcommand{\Pr}{\mathds{P}}
\renewcommand{\E}{\mathds{E}}

\begin{document}


\title*{Statistical Distances and Their Role in Robustness}

\author{Marianthi Markatou, Yang Chen, Georgios Afendras and Bruce G. Lindsay}

\authorrunning{Markatou et al.}

\institute{Marianthi Markatou \at Department of Biostatistics, University at Buffalo, Buffalo, NY, 14214, \\\email{markatou@buffalo.edu}
\and Yang Chen \at Department of Biostatistics, University at Buffalo, Buffalo, NY, 14214, \\\email{ychen57@buffalo.edu}
\and Georgios Afendras \at Department of Biostatistics, University at Buffalo, Buffalo, NY, 14214, \\\email{gafendra@buffalo.edu}
\and Bruce G. Lindsay \at Department of Statistics, The Pennsylvania State University, University Park, PA, 16820, \\\email{bgl@psu.edu}}

\maketitle


\abstract*{Statistical distances, divergences, and similar quantities have a large history and play a fundamental role in statistics, machine learning and associated scientific disciplines. However, within the statistical literature, this extensive role has too often been played out behind the scenes, with other aspects of the statistical problems being viewed as more central, more interesting, or more important. The behind the scenes role of statistical distances shows up in estimation, where we often use estimators based on minimizing a distance, explicitly or implicitly, but rarely studying how the properties of a distance determine the properties of the estimators. Distances are also prominent in goodness-of-fit, but the usual question we ask is \textquotedblleft how powerful is this method against a set of interesting alternatives\textquotedblright \hspace{0.01cm} not \textquotedblleft what aspect of the distance between the hypothetical model and the alternative are we measuring?\textquotedblright \newline\indent
Our focus is on describing the statistical properties of some of the distance measures we have found to be most important and most visible. We illustrate the robust nature of Neyman's chi-squared and the non-robust nature of Pearson's chi-squared statistics and discuss the concept of discretization robustness.}

\abstract{Statistical distances, divergences, and similar quantities have a large history and play a fundamental role in statistics, machine learning and associated scientific disciplines. However, within the statistical literature, this extensive role has too often been played out behind the scenes, with other aspects of the statistical problems being viewed as more central, more interesting, or more important. The behind the scenes role of statistical distances shows up in estimation, where we often use estimators based on minimizing a distance, explicitly or implicitly, but rarely studying how the properties of a distance determine the properties of the estimators. Distances are also prominent in goodness-of-fit, but the usual question we ask is \textquotedblleft how powerful is this method against a set of interesting alternatives\textquotedblright \hspace{0.01cm} not \textquotedblleft what aspect of the distance between the hypothetical model and the alternative are we measuring?\textquotedblright \newline\indent
Our focus is on describing the statistical properties of some of the distance measures we have found to be most important and most visible. We illustrate the robust nature of Neyman's chi-squared and the non-robust nature of Pearson's chi-squared statistics and discuss the concept of discretization robustness.}


\section{Introduction}
\label{secction1}

Distance measures play a ubiquitous role in statistical theory and thinking. However, within the statistical literature this extensive role has too often been played out behind the scenes, with other aspects of the statistical problems being viewed as more central, more interesting, or more important.

The behind the scenes role of statistical distances shows up in estimation, where we often use estimators based on minimizing a distance, explicitly or implicitly, but rarely studying how the properties of the distance determine the properties of the estimators. Distances are also prominent in goodness-of-fit (GOF) but the usual question we ask is how powerful is our method against a set of interesting alternatives not what aspects of the difference between the hypothetical model and the alternative are we measuring?

How can we interpret a numerical value of a distance? In goodness-of-fit we learn about Kolmogorov-Smirnov and Cram\'er-von Mises distances but how do these compare with each other? How can we improve their properties by looking at what statistical properties are they measuring?

Past interest in distance functions between statistical populations had a two-fold purpose. The first purpose was to prove existence theorems regarding some optimum solutions in the problem of statistical inference. Wald \cite{Wald1950} in his book on statistical decision functions gave numerous definitions of distance between two distributions which he primarily introduced for the purpose of creating decision functions. In this context, the choice of the distance function is not entirely arbitrary, but it is guided by the nature of the mathematical problem at hand.

Statistical distances are defined in a variety of ways, by comparing distribution functions, density functions or characteristic functions / moment generating functions. Further, there are discrete and continuous analogues of distances based on comparing density functions, where the word \textquotedblleft density\textquotedblright \hspace{0.02cm} is used to also indicate probability mass functions. Distances can also be constructed based on the divergence between a nonparametric probability density estimate and a parametric family of densities. Typical examples of distribution-based distances are the Kolmogorov-Smirnov and Cram\'er-von Mises distances. A separate class of distances is based upon comparing the empirical characteristic function with the theoretical characteristic function that corresponds, for example, to a family of models under study, or by comparing empirical and theoretical versions of moment generating functions.

In this paper we proceed to study in detail the properties of some statistical distances, and especially the properties of the class of chi-squared distances. We place emphasis on determining the sense in which we can offer meaningful interpretations of these distances as measures of statistical loss. Section~\ref{secction2} of the paper discusses the definition of a statistical distance in the discrete probability models context. Section~\ref{secction3} presents the class of chi-squared distances and their statistical interpretation again in the context of discrete probability  models. Subsection~\ref{subsection3.3} discusses metric and other properties of the symmetric chi-squared distance. One of the key issues in the construction of model misspecification measures is that allowance should be made for the scale difference between observed data and a hypothesized model continuous distribution. To account for this difference in scale we need the distance measure to exhibit discretization robustness, a concept that is discussed in subsection~\ref{subsection4.1}. To achieve discretization robustness we need sensitive distances, and this requirement dictates a balance of sensitivity and statistical noise. Various strategies that deal with this issue are discussed in the literature and we briefly discuss them in subsection~\ref{subsection4.1}. A flexible class of distances that allows the user to adjust the noise/sensitivity trade-off is the kernel smoothed distances upon which we briefly remark on in section~\ref{secction4}. Finally, section~\ref{secction5} presents further discussion.


\section{The Discrete Setting}
\label{secction2}

Procedures based on minimizing the distance between two density functions express the idea that a fitted statistical model should summarize reasonably well the data and that assessment of the adequacy of the fitted model can be achieved by using the value of the distance between the data and the fitted model.

The essential idea of density-based minimum distance methods has been presented in the literature for quite some time as it is evidenced by the method of minimum chi-squared \cite{Neyman1949}. An extensive list of minimum chi-squared methods can be found in Berkson \cite{Berkson1980}. Matusita \cite{Matusita1955} and Rao \cite{Rao1963} studied minimum Hellinger distance estimation in discrete models while Beran \cite{Beran1977} was the first to use the idea of minimum Hellinger distance in continuous models.

We begin within the discrete distribution framework so as to provide the clearest possible focus for our interpretations. Thus, let $\mathscr{T}=\{0, 1, 2, \cdots, T\}$, where $T$ is possibly infinite, be a discrete sample space. On this sample space we define a true probability density $\tau(t)$, as well as a family of densities $\mathscr{M}=\{m_{\theta}(t): \theta \in \Theta \}$, where $\Theta$ is the parameter space. Assume we have independent and identically distributed random variables $X_{1}, X_{2}, \cdots, X_{n}$ producing the realizations $x_{1}, x_{2}, \cdots, x_{n}$ from $\tau(\cdot)$. We record the data as $d(t)={n(t)}/{n}$, where $n(t)$ is the number of observations in the sample with value equal to $t$.

\begin{definition}
\label{definition1}
We will say that $\rho(\tau, m)$ is a statistical distance between two probability distributions with densities $\tau$, $m$ if $\rho(\tau, m) \geq 0$, with equality if and only if $\tau$ and $m$ are the same for all statistical purposes.
\end{definition}

Note that we do not require symmetry or the triangle inequality, so that $\rho(\tau, m)$ is not formally a metric. This is not a drawback as well known distances, such as Kullback-Leibler, are not symmetric and do not satisfy the triangle inequality.

We can extend the definition of a distance between two densities to that of a distance between a density and a class of densities as follows.

\begin{definition}
\label{definition2}
let $\mathscr{M}$ be a given model class and $\tau$ be a probability density that does not belong in the model class $\mathscr{M}$. Then, the distance between $\tau$ and $\mathscr{M}$ is defined as
\begin{equation*}
\rho(\tau, \mathscr{M})=\inf_{m \in \mathscr{M}}\rho(\tau, m),
\end{equation*}
whenever the infimum exists. Let $m_{best} \in \mathscr{M}$ be the best fitting model, then
\begin{equation*}
\rho(\tau, m_{best}) \triangleq \rho(\tau, \mathscr{M}).
\end{equation*}
\end{definition}

We interpret $\rho(\tau, m)$ or $\rho(\tau, \mathscr{M})$ as measuring the \textquotedblleft lack-of-fit\textquotedblright \hspace{0.02cm} in the sense that larger values of $\rho(\tau, m)$ mean that the model element $m$ is a worst fit to $\tau$ for our statistical purposes. Therefore, we will require $\rho(\tau, m)$ to indicate the worst mistake that we can make if we use $m$ instead of $\tau$. The precise meaning of this statement will be obvious in the case of the total variation distance, as we will see that the total variation distance measures the error, in probability, that is made when $m$ is used instead of $\tau$.

Lindsay \cite{Lindsay1994} studied the relationship between the concepts of efficiency and robustness for the class of $f$- or $\phi$-divergences in the case of discrete probability models and defined the concept of Pearson residuals as follows.

\begin{definition}
\label{definition3}
For a pair of densities $\tau$, $m$ define the Pearson residual by
\begin{equation}
\label{equation1}
\delta(t) = \frac{\tau(t)}{m(t)}-1,
\end{equation}
with range the interval $[-1, \infty)$.
\end{definition}

This residual has been used by Lindsay \cite{Lindsay1994}, Basu and Lindsay \cite{BasuLindsay1994}, Markatou \cite{Markatou2000, Markatou2001}, and Markatou et al. \cite{Markatou1997, Markatou1998} in investigating the robustness of the minimum disparity and weighted likelihood estimators respectively. It also appears in the definition of the class of power divergence measures defined by
\begin{align*}
\rho(\tau, m) & = \frac{1}{\lambda(\lambda+1)} \sum \tau(t) \left\{ \left(\frac{\tau(t)}{m(t)}\right)^{\lambda} - 1 \right\} \\
& = \frac{1}{\lambda(\lambda+1)} \sum m(t) \{ (1 + \delta(t))^{\lambda + 1} - 1 \}.
\end{align*}
For $\lambda = -2, -1, -{1}/{2}, 0$ and $1$ one obtains the well-known Neyman's chi-squared (divided by 2) distance, Kullback-Leibler divergence, twice-squared Hellinger distance, likelihood disparity and Pearson's chi-squared (divided by 2) distance respectively. For additional details see Lindsay \cite{Lindsay1994} and Basu and Lindsay \cite{BasuLindsay1994}.

A special class of distance measures we are particularly interested in is the class of chi-squared measures. In what follows we discuss in detail this class.


\section{Chi-Squared Distance Measures}
\label{secction3}

We present the class of chi-squared disparities and discuss their properties. We offer loss function interpretations of the chi-squared measures and show that Pearson's chi-squared is the supremum of squared $Z$-statistics while Neyman's chi-squared is the supremum of squared $t$-statistics. We also show that the symmetric chi-squared is a metric and offer a testing interpretation for it.

We start with the definition of a generalized chi-squared distance between two densities $\tau$, $m$.

\begin{definition}
\label{definition4}
Let $\tau(t)$, $m(t)$ be two discrete probability distributions. Then, define the class of generalized chi-squared distances as
\begin{equation*}
\chi_{a}^{2}(\tau, m) = \sum \frac{[\tau(t) - m(t)]^{2}}{a(t)},
\end{equation*}
where $a(t)$ is a probability mass function.
\end{definition}

Notice that if we restrict ourselves to the multinomial setting and  choose $\tau(t) = d(t)$ and $a(t) = m(t)$, the resulting chi-squared distance is Pearson's chi-squared statistic. Lindsay \cite{Lindsay1994} studied the robustness properties of a version of $\chi_{a}^{2}(\tau, m)$ by taking $a(t) = [\tau(t) + m(t)]/2$. The resulting distance is called \textit{symmetric chi-squared}, and it is given as
\begin{equation*}
S^{2}(\tau, m) = \sum \frac{2[\tau(t) - m(t)]^{2}}{\tau(t) + m(t)}.
\end{equation*}

The chi-squared distance is symmetric because $S^{2}(\tau, m) = S^{2}(m, \tau)$ and satisfies the triangle inequality. Thus, by definition it is a proper metric, and there is a strong dependence of the properties of the distance on the denominator $a(t)$. In general we can use as a denominator $a(t) = \alpha\tau(t) + \overline{\alpha}m(t)$, $\overline{\alpha} = 1- \alpha$, $\alpha \in [0, 1]$. The so defined distance is called blended chi-squared \cite{Lindsay1994}.


\subsection{Loss Function Interpretation}
\label{subsection3.1}

We now discuss the loss function interpretation of the aforementioned class of distances.

\begin{proposition}
\label{proposition1}
Let $\tau$, $m$ be two discrete probabilities. Then
\begin{equation*}
\rho(\tau, m) = \sup_{h} \frac{\{ \E_{\tau}(h(X)) - \E_{m}(h(X)) \}^{2}}{\Var_{a}(h(X))},
\end{equation*}
where $a(t)$ is a density function, and $h(X)$ has finite second moment.
\end{proposition}

\begin{proof}
\label{proof_of_proposition1}
\smartqed
Let $h$ be a function defined on the sample space. We can prove the above statement by looking at the equivalent problem
\begin{equation*}
\sup_{h} \{ \E_{\tau}(h(X)) - \E_{m}(h(X)) \}^{2}, \quad\text{subject to} \ \Var_{a}(h(X)) =1.
\end{equation*}

Note that the transformation from the original problem to the simpler problem stated above is without loss of generality because the first problem is scale invariant, that is, the functions $\widehat{h}$ and $c\widehat{h}$ where $c$ is a constant give exactly the same values. In addition, we have location invariance in that $h(X)$ and $h(X) + c$ give again the same values, and symmetry requires us to solve
\begin{equation*}
\sup_{h} \{ \E_{\tau}(h(X)) - \E_{m}(h(X)) \},\quad \text{subject to} \ \sum h^{2}(t) a(t) = 1.
\end{equation*}

To solve this linear problem with its quadratic constraint we use Lagrange multipliers. The Lagrangian is given as
\begin{equation*}
L(t) = \sum h(t) (\tau(t) - m(t)) - \lambda \left\{ \sum h^{2}(t) a(t) - 1 \right\}.
\end{equation*}

Then
\begin{equation*}
\frac{\partial}{\partial h} L(t) = 0, \text{\ for each value of $t$},
\end{equation*}
is equivalent to
\begin{equation*}
\tau(t) - m(t) - 2\lambda h(t) a(t) = 0, \forall t,
\end{equation*}
or
\begin{equation*}
\widehat{h}(t) = \frac{\tau(t) - m(t)}{2 \lambda a(t)}.
\end{equation*}

Using the constraint we obtain
\begin{align*}
\sum \frac{ [ \tau(t) - m(t) ]^{2} }{4 \lambda^{2} a(t)} =1   \Rightarrow \widehat{\lambda} = \frac{1}{2} \left\{ \sum \frac{ [ \tau(t) - m(t) ]^{2} }{a(t)} \right\}^{1/2}.
\end{align*}

Therefore,
\begin{equation*}
\widehat{h} (t) = \frac{\tau(t) - m(t)}{a(t)  \sqrt{ \sum \frac{ [ \tau(t) - m(t) ]^{2} }{a(t)} } }.
\end{equation*}

If we substitute the above value of $h$ in the original problem we obtain
\[
\begin{split}
\sup_{h} \{ \E_{\tau}(h(X)) - &\E_{m}(h(X)) \}^{2}
 = \sup_{h} \left\{  \sum h(t) [ \tau(t) - m(t) ]^{2} \right\} \\
&= \left\{ \sum \frac{[\tau(t) - m(t)]^{2}}{a(t)  \sqrt{ \sum \frac{ [ \tau(t) - m(t) ]^{2} }{a(t)} } } \right\}^{2}\\
&= \left\{ \frac{1}{\sqrt{ \sum \frac{ [ \tau(t) - m(t) ]^{2} }{a(t)} }} (  \sum \frac{[\tau(t) - m(t)]^{2}}{a(t)} ) \right\}^{2}
 = \sum \frac{ [ \tau(t) - m(t) ]^{2} }{a(t)},
\end{split}
\]
as was claimed.
\qed
\end{proof}

\begin{remark}
\label{remark1}
Note that $\widehat{h}(t)$ is the least favorable function for detecting differences between means of two distributions.
\end{remark}

\begin{corollary}
\label{corollary1}
The standardized function which creates the largest difference in means is
\begin{equation*}
\widehat{h}(t) =  \frac{\tau(t) - m(t)}{a(t) \sqrt{\chi_{a}^{2}}},
\end{equation*}
where $\chi_{a}^{2} = \sum \frac{[\tau(t) - m(t)]^{2}}{a(t)}$, and the corresponding difference in means is
\begin{equation*}
\E_{\tau}[\widehat{h}(t)] - \E_{m}[\widehat{h}(t)] = \sqrt{\chi_{a}^{2}}.
\end{equation*}
\end{corollary}

\begin{remark}
\label{remark2}
Are there any additional distances that can be obtained as solutions to an optimization problem? And what is the statistical interpretation of these optimization problems? To answer the aforementioned questions we first present the optimization problems associated with the Kullback-Leibler and Hellinger distances. In fact, the entire class of the blended weighted Hellinger distances can be obtained as a solution to an appropriately defined optimization problem. Secondly, we discuss the statistical interpretability of these problems by connecting them, by analogy, to the construction of confidence intervals via Scheff\'e's method.
\end{remark}

\begin{definition}
\label{definition5}
The Kullback-Leibler divergence or distance between two discrete probability density functions is defined as
\begin{equation*}
KL(\tau, m_{\beta}) = \sum_{x} m_{\beta} (x) [ \log m_{\beta} (x) - \log \tau(x) ].
\end{equation*}
\end{definition}

\begin{proposition}
\label{proposition2}
The Kullback-Leibler distance is obtained as a solution of the optimization problem
\begin{equation*}
\sup_{h} \sum_{x} h(x) m_{\beta} (x), \quad\text{subject to} \ \sum_{x} e^{h(x)} \tau(x) \leq 1,
\end{equation*}
where $h(\cdot)$ is a function defined on the same space as $\tau$.
\end{proposition}

\begin{proof}
\label{proof_of_proposition2}
\smartqed
It is straightforward if one writes the Lagrangian and differentiates with respect to $h$.
\qed
\end{proof}

\begin{definition}
\label{definition6}
The class of squared blended weighted Hellinger distances ($BWHD_{\alpha}$) is defined as
\begin{equation*}
(BWHD_{\alpha})^{2} = \sum_{x} \frac{ [ \tau(x) - m_{\beta}(x) ]^{2} }{ 2\left[ \alpha\sqrt{\tau(x)} + \overline{\alpha}\sqrt{m_{\beta}(x)} \right]^{2} },
\end{equation*}
where $0 < \alpha < 1$, $\overline{\alpha} = 1- \alpha$ and $\tau(x)$, $m_{\beta}(x)$ are two probability densities.
\end{definition}

\begin{proposition}
\label{proposition3}
The class of $BWHD_{\alpha}$ arises as a solution to the optimization problem
\begin{equation*}
\sup_{h} \sum_{x} h(x) [ \tau(x) - m_{\beta}(x) ], \quad\text{subject to} \ \sum_{x} h^{2}(x) \left[ \alpha\sqrt{\tau(x)} + \overline{\alpha}\sqrt{m_{\beta}(x)} \right]^{2} \leq 1.
\end{equation*}
When $\alpha = \overline{\alpha} = {1}/{2}$, the $(BWHD_{1/2})^{2}$ gives twice the squared Hellinger distance.
\end{proposition}

\begin{proof}
\label{proof_of_proposition3}
\smartqed
Straightforward.
\qed
\end{proof}

Although both Kullback-Leibler and blended weighted Hellinger distances are solutions of appropriate optimization problems, they do not arise from optimization problems in which the constraints can be interpreted as variances. To exemplify and illustrate further this point we first need to discuss the connection with Scheff\'e's confidence intervals.

\bigbreak
One of the methods of constructing confidence intervals is Scheff\'e's method. The method adjusts the significance levels of the confidence intervals for general contrasts to account for multiple comparisons. The procedure, therefore, controls the overall significance for any possible contrast or set of contrasts and can be stated as follows,
\begin{equation*}
\sup_{\pmb{c}} \left| \pmb{c}^{T} (\pmb{y} - \pmb{\mu}) \right| < K \widehat{\sigma}, 
\quad\text{subject to \ } \parallel \pmb{c} \parallel =1, \ \pmb{c}^{T}\pmb{1} = 0,
\end{equation*}
where $\widehat{\sigma}$ is an estimated contrast variance, $K$ is an appropriately defined constant.

The chi-squared distances extend this framework as follows. Assume that $\mathscr{H}$ is a class of functions which are taken, without loss of generality, to have zero expectation. Then, we construct the optimization problem $\sup_{h} \int h(x) [ \tau(x) - m_{\beta}(x) ] dx$, subject to a constraint that can possibly be interpreted as a constraint on the variance of $h(x)$ either under the hypothesized model distribution or under the distribution of the data.

The chi-squared distances arise as solutions of optimization problems subject to variance constrains. As such, they are interpretable as tools that allow the construction of \textquotedblleft Scheff\'e-type\textquotedblright confidence intervals for models. On the other hand, distances such as the Kullback-Leibler or  blended weighted Hellinger distance do not arise as solutions of optimization problems subject to interpretable variance constraints. As such they cannot be used to construct confidence intervals for models.


\subsection{Loss Analysis of Pearson and Neyman Chi-Squared Distances}
\label{subsection3.2}

We next offer interpretations of the Pearson chi-squared and Neyman chi-squared statistics. These interpretations are not well known; furthermore, they are useful in illustrating the robustness character of the Neyman statistic and the non-robustness character of the Pearson statistic.

Recall that the Pearson statistic is
\begin{align*}
\sum \frac{ [d(t) - m(t)]^{2} }{m(t)} & = \sup_{h} \frac{ [ \E_{d}(h(X)) - \E_{m}(h(X)) ]^{2} }{ \Var_{m}(h(X)) } \\
& = \frac{1}{n} \sup_{h} \frac{ [ \frac{1}{n} \sum h(X_{i}) - \E_{m}(h(X)) ]^{2} }{ \frac{1}{n} \Var_{m}(h(X)) }
 = \frac{1}{n} \sup_{h} Z_{h}^{2},
\end{align*}
that is, the Pearson statistic is the supremum of squared $Z$-statistics.

A similar argument shows that Neyman's chi-squared equals $ \sup_{h} t_{h}^{2}$, the supremum of squared $t$-statistics.

This property shows that the chi-squared measures have a statistical interpretation in that a small chi-squared distance indicates that the means are close on the scale of standard deviation. Furthermore, an additional advantage of the above interpretations is that the robustness character of these statistics is exemplified. Neyman's chi-squared, being the supremum of squared $t$-statistics, is robust, whereas Pearson's chi-squared is non-robust, since it is the supremum of squared $Z$-statistics.

\subruninhead{Signal-to-noise:} There is an additional interpretation of the chi-squared statistic that rests on the definition of signal-to-noise ratio that comes from the engineering literature.

Consider the pair of hypotheses $H_{0}: X_{i} \sim \tau$ versus the alternative $H_{1}:  X_{i} \sim m$, where $X_{i}$ are independent and identically distributed random variables. If we consider the set of randomized test functions that depend on the \textquotedblleft output\textquotedblright \hspace{0.02cm} function $h$, the distance between $H_{0}$ and $H_{1}$ is
\begin{equation*}
S^{2}(\tau, m) = \frac{[\E_{m}(h(X))-\E_{\tau}(h(X))]^{2}}{\Var_{\tau}(h(X))}.
\end{equation*}

This quantity is a generalization of one of the more common definitions of signal-to-noise ratio. If, instead of working with a given output function $h$, we take supremum over the output functions $h$, we obtain Neyman's chi-squared distance, which has been used in the engineering literature for robust detection. Further, the quantity $S^{2}(\tau, m)$ has been used in the design of decision systems \cite{Poor1980}.


\subsection{Metric Properties of the Symmetric Chi-Squared Distance}
\label{subsection3.3}

The symmetric chi-squared distance, defined as
\begin{equation*}
S^{2}(\tau, m) = \sum \frac{ 2[\tau(t) - m(t)]^{2} }{m(t) + \tau(t) },
\end{equation*}
can be viewed as a good compromise between the non-robust Pearson distance and the robust Neyman distance. In what follows, we prove that $S^{2}(\tau, m)$ is indeed a metric. The following series of lemmas will help us establish the triangle inequality for $S^{2}(\tau, m)$.

\begin{lemma}
\label{lemma1}
If $a$, $b$, $c$ are numbers such that $0 \leq a \leq b \leq c$ then
\begin{equation*}
\frac{c-a}{\sqrt{c+a}} \leq \frac{b-a}{\sqrt{b+a}} + \frac{c-b}{\sqrt{c+b}}.
\end{equation*}
\end{lemma}

\begin{proof}
\label{proof_lemma1}
\smartqed
First we work with the right-hand side of the above inequality. Write
\begin{align*}
\frac{b-a}{\sqrt{b+a}} + \frac{c-b}{\sqrt{c+b}} &= (c-a) \{ \frac{b-a}{c-a} \frac{1}{\sqrt{b+a}} +  \frac{c-b}{c-a} \frac{1}{\sqrt{c+b}} \} \\
& = (c-a) \{ \frac{\alpha}{\sqrt{a+b}} +  (1-\alpha) \frac{1}{\sqrt{c+b}} \},
\end{align*}
where $\alpha = (b-a)/(c-a)$. Set $g(t) = 1 / \sqrt{t}$, $t>0$. Then $g^{\prime \prime}(t) = \frac{d^{2}}{d t^{2}} g(t) > 0$, hence the function $g(t)$ is convex. Therefore, the aforementioned relationship becomes
\begin{equation*}
(c-a)\{ \alpha g(a+b) + (1-\alpha) g(c+b) \}.
\end{equation*}

But
\begin{equation*}
\alpha g(a+b) + (1-\alpha) g(c+b) \geq g(\alpha(a+b) + (1-\alpha)(c+b)),
\end{equation*}
where
\begin{equation*}
\alpha(a+b) + (1-\alpha)(c+b) = \frac{b-a}{c-a} (a+b) + \frac{c-b}{c-a} (b+c) = c+a.
\end{equation*}

Thus
\begin{equation*}
\alpha g(a+b) + (1-\alpha) g(c+b) \geq g(c+a),
\end{equation*}
and hence
\begin{equation*}
\frac{b-a}{\sqrt{b+a}} + \frac{c-b}{\sqrt{c+b}} \geq \frac{c-a}{\sqrt{c+a}},
\end{equation*}
as was stated.
\qed
\end{proof}

Note that because the function is strictly convex we do not obtain equality except when $a=b=c$.

\begin{lemma}
\label{lemma2}
If $a$, $b$, $c$ are numbers such that $a \geq 0$, $b \geq 0$, $c \geq 0$ then
\begin{equation*}
\left| \frac{c-a}{\sqrt{c+a}} \right|  \leq \left| \frac{b-a}{\sqrt{b+a}} \right| + \left| \frac{c-b}{\sqrt{c+b}} \right|.
\end{equation*}
\end{lemma}

\begin{proof}
\label{proof_lemma2}
\smartqed
We will distinguish three different cases.

Case 1: $0 \leq a \leq b \leq c$ is already discussed in Lemma~\ref{lemma1}.

Case 2: $0 \leq c \leq b \leq a$ can be proved as in Lemma~\ref{lemma1} by interchanging the role of $a$ and $c$.

Case 3: In this case $b$ is not between $a$ and $c$, thus either $a \leq c \leq b $ or $b \leq a \leq c$.

Assume first that $a \leq c \leq b $. Then we need to show that
\begin{equation*}
\frac{c-a}{\sqrt{c+a}} \leq \frac{b-a}{\sqrt{b+a}}.
\end{equation*}

We will prove this by showing that the above expressions are the values of an increasing function at two different points. Thus, consider
\begin{equation*}
f_{1}(t) = \frac{t-a}{\sqrt{t+a}}.
\end{equation*}

It follows that
\begin{equation*}
f_{1}(b) = \frac{b-a}{\sqrt{b+a}} \text{ \ and \ }  f_{1}(c) = \frac{c-a}{\sqrt{c+a}}.
\end{equation*}

The function $f_{1}(t)$ is increasing because $f_{1}^{\prime} > 0$ (recall $a \geq 0$) and since $c \leq b$ this implies $f_{1}(c) \leq f_{1}(b)$. Similarly we prove the inequality for  $b \leq a \leq c$.
\qed
\end{proof}

\begin{lemma}
\label{lemma3}
The triangle inequality holds for the symmetric chi-squared distance $S^{2}(\tau, m)$, that is,
\begin{equation*}
\{ S^{2}(\tau, m) \}^{1/2} \leq \{ S^{2}(\tau, g) \}^{1/2} + \{ S^{2}(g, m) \}^{1/2}.
\end{equation*}
\end{lemma}

\begin{proof}
\label{proof_lemma3}
\smartqed
Set
\begin{eqnarray*}
\alpha_{t} = \frac{ \left| \tau(t) - g(t) \right| }{ \sqrt{ \tau(t) + g(t) } }, \text{ \ }
\beta_{t} = \frac{ \left| g(t) - m(t) \right| }{ \sqrt{ g(t) + m(t) } }.
\end{eqnarray*}

By Lemma~\ref{lemma2}
\begin{equation*}
\left\{ \sum \alpha_{t}^{2} \right\}^{1/2} \leq \left\{ \sum (\alpha_{t} + \beta_{t})^{2} \right\}^{1/2}.
\end{equation*}

But
\begin{align*}
\sum (\alpha_{t} + \beta_{t})^{2} &= \sum \alpha_{t}^{2} + \sum \beta_{t}^{2} + 2\sum \alpha_{t} \beta_{t} \\
& \leq \sum \alpha_{t}^{2} + \sum \beta_{t}^{2} + 2 \left\{ \sum \alpha_{t}^{2} \right\}^{1/2} \left\{ \sum \beta_{t}^{2} \right\}^{1/2}.
\end{align*}

Therefore
\begin{equation*}
\sum (\alpha_{t} + \beta_{t})^{2} \leq \left\{ \sqrt{\sum \alpha_{t}^{2}} + \sqrt{\sum \beta_{t}^{2}} \right\}^{2},
\end{equation*}
and hence
\begin{equation*}
\left\{ \sum (\alpha_{t} + \beta_{t})^{2} \right\}^{1/2} \leq \left\{ \sum \alpha_{t}^{2} \right\}^{1/2} + \left\{ \sum \beta_{t}^{2} \right\}^{1/2},
\end{equation*}
as was claimed.
\qed
\end{proof}

\begin{remark}
\label{remark3}
The inequalities proved in Lemma~\ref{lemma1} and \ref{lemma2} imply that if $\tau \neq m$ there is no \textquotedblleft straight line\textquotedblright connecting $\tau$ and $m$, in that there does not exist $g$ between $\tau$ and $m$ for which the triangle inequality is an equality.
\end{remark}

Therefore, the following proposition holds.

\begin{proposition}
\label{proposition4}
The symmetric chi-squared distance $S^{2}(\tau, m)$ is indeed a metric.
\end{proposition}

\runinhead{\textit{A testing interpretation of the symmetric chi-squared distance}:} let $\phi$ be a test function and consider the problem of testing the null hypothesis that the data come from a density $f$ versus the alternative that the data come from $g$. Let $\theta$ be a random variable with value $1$ if the alternative is true and $0$ if the null hypothesis is true. Then

\begin{proposition}
\label{proposition5}
The solution $\phi_{opt}$ to the optimization problem
\begin{equation*}
\min_{\phi} \E_{\pi} [ ( \theta - \phi(x))^{2} ],
\end{equation*}
where $\pi(\theta)$ is the prior probability on $\theta$, given as
\begin{equation*}
\pi(\theta) = \left\{
\begin{array}{r@{\;\;}l}
& 1/2,\quad \text{if $\theta=0$}\\
& 1/2,\quad \text{if $\theta=1$}
\end{array}
\right. ,
\end{equation*}
is not a $0-1$ decision, but equals the posterior expectation of $\theta$ given $X$. That is
\begin{align*}
\phi(t) & = \E(\theta \mid X=t) = \Pr(\theta =1 \mid X=t)
= \frac{ \frac{1}{2}g(t) }{ \frac{1}{2}f(t) + \frac{1}{2}g(t) } = \frac{ g(t) }{ f(t) + g(t) },
\end{align*}
the posterior probability that the alternative is correct.
\end{proposition}

\begin{proof}
\label{proof_of_proposition5}
\smartqed
We have
\begin{equation*}
\E(\theta \mid X) = \frac{1}{2} \E_{H_{1}} [ ( 1-\phi )^{2} ] + \frac{1}{2} \E_{H_{0}} (\phi^{2}).
\end{equation*}

But
\begin{equation*}
\E_{H_{1}} [ ( 1-\phi (X) )^{2} ] = \sum_{t} ( 1-\phi(t) )^{2} g(t),
\end{equation*}
and
\begin{equation*}
\E_{H_{0}} (\phi^{2}(X)) = \sum_{t}\phi^{2}(t)f(t),
\end{equation*}
hence
\begin{equation*}
\phi_{opt}(t) = \frac{ g(t) }{ f(t) + g(t) },
\end{equation*}
as was claimed.
\qed
\end{proof}

\begin{corollary}
\label{corollary2}
The minimum risk is given as
\begin{equation*}
\frac{1}{4} \left(1- \frac{S^{2}}{4} \right),
\end{equation*}
where
\begin{equation*}
S^{2} = S^{2}(f, g) = \sum \frac{ [ f(t) - g(t) ]^{2} }{ \frac{1}{2}f(t) + \frac{1}{2} g(t) }.
\end{equation*}
\end{corollary}

\begin{proof}
\label{proof_of_corollary2}
\smartqed
Substitute $\phi_{opt}$ in $\E_{\pi}[ ( \theta - \phi )^{2} ]$ to obtain
\begin{equation*}
\E_{\pi}[ ( \theta - \phi_{opt} )^{2} ] =\frac{1}{2} \sum_{t} \frac{f(t) g(t)}{ f(t) + g(t) }.
\end{equation*}

Now set
\begin{align*}
A = \sum_{t} \frac{ [f(t) + g(t) ]^{2} }{ f(t) + g(t)} = 2, \text{ \ }
B = \sum_{t} \frac{ [ f(t) - g(t) ]^{2} }{ f(t) + g(t)} = \frac{1}{2} S^{2}.
\end{align*}

Then
\begin{equation*}
A-B = 4 \sum_{t} \frac{f(t) g(t)}{ f(t) + g(t) } = 2 - \frac{1}{2} S^{2},
\end{equation*}
or, equivalently,
\begin{equation*}
\sum_{t} \frac{f(t) g(t)}{ f(t) + g(t) } = \frac{1}{4} \left(2 - \frac{1}{2} S^{2} \right).
\end{equation*}

Therefore
\begin{align*}
\E_{\pi}[ ( \theta - \phi_{opt} )^{2} ] & = \frac{1}{2} \sum_{t} \frac{f(t) g(t)}{ f(t) + g(t) }
 = \frac{1}{4} \left(1- \frac{S^{2}}{4} \right),
\end{align*}
as was claimed.
\qed
\end{proof}

\begin{remark}
\label{remark4}
Note that $S^{2}(f, g)$ is bounded above by $4$; it becomes equal to $4$ when $f$, $g$ are mutually singular.
\end{remark}

The Kullback-Leibler and Hellinger distances are extensively used in the literature. Yet, we argue that, because they are obtained as solutions to optimization problems with non-interpretable (statistically) constraints, are not appropriate for our purposes. However, we note here that the Hellinger distance is closely related to the symmetric chi-squared distance, although this is not immediately obvious. We elaborate on this statement below.

\begin{definition}
\label{definition7}
Let $\tau$, $m$ be two probability mass functions. The squared Hellinger distance is defined as
\begin{equation*}
H^{2} (\tau, m) = \frac{1}{2} \sum_{x} \left[ \sqrt {\tau(x)} - \sqrt{m(x)} \right]^{2}.
\end{equation*}
\end{definition}

We can more readily see the relationship between the Hellinger and chi-squared distances if we rewrite $H^{2} (\tau, m)$ as
\begin{equation*}
H^{2} (\tau, m) = \frac{1}{2} \sum_{x} \frac{ [ \tau(x) - m(x)]^{2} }{ [ \sqrt {\tau(x)} + \sqrt{m(x)}]^{2} }.
\end{equation*}

\begin{lemma}
\label{lemma4}
The Hellinger distance is bounded by the symmetric chi-squared distance, that is,
\begin{equation*}
\frac{1}{8} S^{2} \leq H^{2} \leq \frac{1}{4} S^{2},
\end{equation*}
where $S^{2}$ denotes the symmetric chi-squared distance.
\end{lemma}

\begin{proof}
\label{proof_of_lemma45}
\smartqed
Note that
\begin{align*}
\left( \sqrt{\tau(x)} + \sqrt{ m(x) } \right)^{2} &= \tau(x) + m(x) + 2 \sqrt{ \tau(x) m(x) }
 \geq \tau(x) + m(x).
\end{align*}

Also
\begin{equation*}
\left( \sqrt{\tau(x)} + \sqrt{ m(x) }  \right)^{2} \leq 2[ \tau(x) + m(x) ],
\end{equation*}
and putting these relationships together we obtain
\begin{align*}
\tau(x) + m(x) & \leq \left( \sqrt{\tau(x)} + \sqrt{ m(x) } \right)^{2}
 \leq 2 [ \tau(x) + m(x) ].
\end{align*}

Therefore
\begin{equation*}
H^{2}(\tau, m) \leq \frac{1}{2} \sum \frac{ [ \tau(x)-m(x) ]^{2} }{ \tau(x)+m(x) } = \frac{1}{4} S^{2},
\end{equation*}
and
\begin{equation*}
H^{2}(\tau, m) \geq \frac{1}{2} \sum \frac{ [ \tau(x)-m(x) ]^{2} }{ 2[\tau(x)+m(x)] } = \frac{1}{8} S^{2},
\end{equation*}
and so
\begin{equation*}
\frac{1}{8} S^{2} (\tau, m) \leq H^{2} (\tau, m) \leq \frac{1}{4} S^{2} (\tau, m),
\end{equation*}
as was claimed.
\qed
\end{proof}


\subsection{Locally Quadratic Distances}
\label{subsection3.4}

A generalization of the chi-squared distances is offered by the locally quadratic distances. We have the following definition.

\begin{definition}
\label{definition8}
A locally quadratic distance between two densities $\tau$, $m$ has the form
\begin{equation*}
\rho(\tau, m) = \sum K_{m} (x, y)[\tau(x)-m(x)][\tau(y)-m(y)],
\end{equation*}
where $K_{m}(x, y)$ is a nonnegative definite kernel, possibly dependent on $m$, and such that
\begin{equation*}
\sum_{x, y} a(x)K_{m}(x, y)a(y) \geq 0,
\end{equation*}
for all functions $a(x)$.
\end{definition}

\begin{example}
\label{example1}
The Pearson distance can be written as
\begin{align*}
\sum \frac{ ( d(t) - m(t) )^{2} }{m(t)} &= \sum \frac{\mathds{1} [s=t]}{ \sqrt{m(s)m(t)}} [d(s) - m(s) ][d(t) - m(t) ] \\
&= \sum K_{m} (s, t) [d(s) - m(s) ][d(t) - m(t) ],
\end{align*}
where $\mathds{1} (\cdot)$ is the indicator function. It is a quadratic distance with kernel
\begin{equation*}
K_{m}(s, t) = \frac{\mathds{1} [s=t]}{ \sqrt{m(s)m(t)}}.
\end{equation*}
\end{example}

\subruninhead{Sensitivity and Robustness:} In the classical robustness literature one of the attributes that a method should exhibit so as to be characterized as robust is the attribute of being resistant, that is insensitive, to the presence of a moderate number of outliers and to inadequacies in the assumed model.

Similarly here, to characterize a statistical distance as robust it should be insensitive to small changes in the true density, that is, the value of the distance should not be greatly affected by small changes that occur in $\tau$. Lindsay \cite{Lindsay1994}, Markatou \cite{Markatou2000, Markatou2001}, and Markatou et al. \cite{Markatou1997, Markatou1998} based the discussion of robustness of the distances under study on a mechanism that allows the identification of distributional errors, that is, on the Pearson residual. A different system of residuals is the set of symmetrized residuals defined as follows.

\begin{definition}
\label{definition9}
If $\tau$, $m$ are two densities the symmetrized residual is defined as
\begin{equation*}
r_{sym}(t) = \frac{ \tau(t) - m(t)}{ \tau(t) + m(t) }.
\end{equation*}
\end{definition}

The symmetrized residuals have range $[-1, 1]$, with value $-1$ when $\tau(t) = 0$ and value 1 when $m(t) = 0$. Symmetrized residuals are important because they allow us to understand the way different distances treat different distributions.

The symmetric chi-squared distance can be written as a function of the symmetrized residuals as follows
\begin{align*}
S^{2}(\tau, m) &= 4 \sum \left(\frac{1}{2}\tau(t) + \frac{1}{2}m(t) \right) \left\{ \frac{ \tau(t) - m(t) }{ \tau(t)  + m(t) } \right\}^{2}
 = 4 \sum b(t) r_{sym}^{2}(t),
\end{align*}
where $b(t) = [\tau(t) + m(t)]/2 $.

The aforementioned expression of the symmetric chi-squared distance allows us to obtain inequalities between $S^{2}(\tau, m)$ and other distances.

\bigbreak
A third residual system is the set of logarithmic residuals, defined as follows.

\begin{definition}
\label{definition10}
Let $\tau$, $m$ be two probability mass functions. Define the logarithmic residuals as
\begin{equation*}
\delta(t) = \log \left( \frac{ \tau(t) }{ m(t) } \right),
\end{equation*}
with $\delta \in (-\infty, \infty) $.
\end{definition}

A value of this residual close to $0$ indicates agreement between $\tau$ and $m$. Large positive or negative values indicate disagreement between the two models $\tau$ and $m$.

\bigbreak
In an analysis of a given data set, there are two types of observations that cause concern: outliers and influential observations. In the literature, the concept of an outlier is defined as follows.

\begin{definition}
\label{definition11}
We define an outlier to be an observation (or a set of observations) which appears to be inconsistent with the remaining observations of the data set.
\end{definition}

Therefore, the concept of an outlier may be viewed in relative terms. Suppose we think a sample arises from a standard normal distribution. An observation from this sample is an outlier if it is somehow different in relation to the remaining observations that were generated from the postulated standard normal model. This means that, an observation with value $4$ may be surprising in a sample of size $10$, but is less so if the sample size is $10000$. In our framework therefore, the extent to which an observation is an outlier depends on both the sample size and the probability of occurrence of the observation under the specified model.

\begin{remark}
\label{remark5}
Davies and Gather \cite{DaviesGather1993} state that although detection of outliers is a topic that has been extensively addressed in the literature, the word \textquotedblleft outlier\textquotedblright was not given a precise definition. Davies and Gather \cite{DaviesGather1993} formalized this concept by defining outliers in terms of their position relative to a central model, and in relationship to the sample size. Further details can be found in their paper.
\end{remark}

On the other hand, the literature provides the following definition of an influential observation.

\begin{definition}
\label{definition12}
(Belsley et al. \cite{Belsley1980}) An influential observation is one which, either individually or together with several other observations, has a demonstrably larger impact on the calculated values of various estimates than is the case for most of the other observations.
\end{definition}

Chatterjee and Hadi \cite{ChatterjeeHadi1986} use this definition to address questions about measuring influence and discuss the different measures of influence and their inter-relationships.

The aforementioned definition is subjective, but it implies that one can order observations in a sensible way according to some measure of influence. Outliers need not be influential observations and influential observations need not be outliers. Large Pearson residuals correspond to observations that are \textit{surprising}, in the sense that they occur in locations with small model probability. This is different from influential observations, that is from observations for which their presence or absence greatly affects the value of the maximum likelihood estimator.

Outliers can be surprising observations as well as influential observations. In a normal location-scale model, an outlying observation is both surprising and influential on the maximum likelihood estimator of location. But in the double exponential location model, an outlying observation is possible to be surprising but never influential on the maximum likelihood estimator of location as it equals the median.

Lindsay \cite{Lindsay1994} shows that the robustness of these distances is expressed via a key function called \textit{residual adjustment function} (RAF). Further, he studied the characteristics of this function and showed that an important class of RAFs is given by $A_{\lambda}(\delta)=\frac{ (1+\delta)^{\lambda} -1 }{ \lambda + 1 }$, where $\delta$ is the Pearson residual (defined by equation~(\ref{equation1})). From this class we obtain many RAFs; in particular, when $\lambda = -2$ we obtain the RAF corresponding to Neyman's chi-squared distance. For details, see Lindsay \cite{Lindsay1994}.


\section{The Continuous Setting}
\label{secction4}

Our goal is to use statistical distances to construct model misspecification measures. One of the key issues in the construction of misspecification measures in the case of data being realizations of a random variable that follows a continuous distribution is that allowances should be made for the scale difference between observed data and hypothesized model. That is, data distributions are discrete while the hypothesized model is continuous. Hence, we require the distance to exhibit discretization robustness, so it can account for the difference in scale.

To achieve discretization robustness, we need a sensitive distance, which implies a need to balance sensitivity and statistical noise. We will briefly review available strategies to deal with the problem of balancing sensitivity of the distance and statistical noise.

In what follows, we discuss desirable characteristics we require our distance measures to satisfy.


\subsection{Desired Features}
\label{subsection4.1}

\subruninhead{Discretization Robustness:} Every real data distribution is discrete, and therefore is different from every continuous distribution. Thus, a reasonable distance measure must allow for discretization, by saying that the discretized version of a continuous distribution must get closer to the continuous distribution as the discretization gets finer.

A second reason for requiring discretization robustness is that we will want to use the empirical distribution to estimate the true distribution, but without this robustness, there is no hope that the discrete empirical distribution will be closed to any model point.

\subruninhead{The Problem of Too Many Questions:} Thus, to achieve discretization robustness, we need to construct a sensitive distance. This requirement dictates us to carry out a delicate balancing act between sensitivity and statistical noise.

Lindsay \cite{Lindsay2004} discusses in detail the problem of too many questions. Here we only note that to illustrate the issue Lindsay \cite{Lindsay2004} uses the chi-squared distance and notes that the statistical implications of a refinement in partition are the widening of the sensitivity to model departures in new \textquotedblleft directions\textquotedblright \hspace{0.02cm} but, at the same time, this act increases the statistical noise and therefore decreases the power of the chi-squared test in every existing direction.

There are a number of ways to address this problem, but they all seem to involve a loss of statistical information. This means we cannot ask all model fit questions with optimal accuracy. Two immediate solutions are as follows. First, limit the investigation only to a finite list of questions, essentially boiling down to prioritizing the questions asked of the sample. A number of classical goodness-of-fit tests create exactly such a balance. A second approach to the problem of answering infinitely many questions with only a finite number of data points is through the construction of kernel smoothed density measures. Those measures provide a flexible class of distances that allows for adjusting the sensitivity/noise trade-off. Before we briefly comment on this strategy, we discuss statistical distances between continuous probability distributions.


\subsection{The $\pmb{L_{2}}$-Distance}
\label{subsection4.2}

The $L_{2}$ distance is very popular in density estimation. We show below that this distance is not invariant to one-to-one transformations.

\begin{definition}
\label{definition13}
The $L_{2}$ distance between two probability density functions $\tau$, $m$ is defined as
\begin{equation*}
L_{2}^{2} (\tau, m) = \int [ \tau(x) - m(x) ]^{2} dx.
\end{equation*}
\end{definition}

\begin{proposition}
\label{proposition6}
The $L_{2}$ distance between two probability density functions is not invariant to one-to-one transformations.
\end{proposition}

\begin{proof}
\label{proof_of_proposition6}
\smartqed
Let $Y=a(X)$ be a transformation of $X$, which is one-to-one. Then $x=b(y)$, $b(.)$ is the inverse transformation of $a(.)$, and
\begin{align*}
L_{2}^{2} (\tau_{Y}, m_{Y}) &= \int [ \tau_{Y}(y) - m_{Y}(y) ]^{2} dy \\
&= \int [ \tau_{X}(b(y)) - m_{X}(b(y)) ]^{2} ( b^{\prime}(y) )^{2} dy \\
&= \int [ \tau_{X}(x) - m_{X}(x) ]^{2} ( b^{\prime}(a(x)) )^{2}  a^{\prime}(x) dx \\
&= \int [ \tau_{X}(x) - m_{X}(x) ]^{2} b^{\prime}(a(x)) dx \\
&\neq \int [ \tau_{X}(x) - m_{X}(x) ]^{2} dx = L_{2}^{2} (\tau_{X}, m_{X}).
\end{align*}
Thus, the $L_{2}$ distance is not invariant under monotone transformations.
\qed
\end{proof}

\begin{remark}
\label{remark6}
It is easy to see that the $L_{2}$ distance is location invariant. Moreover, scale changes appear as a constant factor multiplying the $L_{2}$ distance.
\end{remark}


\subsection{The Kolmogorov-Smirnov Distance}
\label{subsection4.3}

We now discuss the Kolmogorov-Smirnov distance used extensively in goodness-of-fit problems, and present its properties.

\begin{definition}
\label{definition14}
The Kolmogorov-Smirnov distance between two cumulative distribution functions $F$, $G$ is defined as
\begin{equation*}
\rho_{KS} (F, G) = \sup_{x} \left| F(x) - G(x) \right|.
\end{equation*}
\end{definition}

\begin{proposition}
\label{propostion7}
(Testing Interpretation) Let $H_{0}: \tau = f$ versus $H_{1}: \tau = g$ and that only test functions $\varphi$ of the form $\mathds{1} (x \leq x_{0})$ or $\mathds{1} (x > x_{0})$ for arbitrary $x_{0}$ are allowed. Then
\begin{equation*}
\rho_{KS} (F, G) = \sup \left| \E_{H_{1}}[\varphi(X)] - \E_{H_{0}}[\varphi(X)] \right|.
\end{equation*}
\end{proposition}

\begin{proof}
\label{proof_of_proposition7}
\smartqed
The difference between power and size of the test is $G(x_{0}) - F(x_{0})$. Therefore,
\begin{align*}
\sup_{x_{0}} \left| G(x_{0}) - F(x_{0}) \right| &= \sup_{x_{0}} \left| F(x_{0}) - G(x_{0}) \right|
=  \rho_{KS} (F, G),
\end{align*}
as was claimed.
\qed
\end{proof}

\begin{proposition}
\label{proposition8}
The Kolmogorov-Smirnov distance is invariant under monotone transformations.
\end{proposition}

\begin{proof}
\label{proof_of_proposition8}
\smartqed
Write
\begin{equation*}
F(x_{0}) - G(x_{0}) = \int \mathds{1} (x \leq x_{0}) [ f(x) - g(x) ] dx.
\end{equation*}

Let $Y=a(X)$ be a one-to-one transformation and $b(\cdot)$ be the corresponding inverse transformation. Then $x=b(y)$ and $dy=a^{\prime}(x)dx$, so
\begin{align*}
F_{Y}(y_{0}) - G_{Y}(y_{0}) &= \int \mathds{1} (y \leq y_{0}) [ f_{Y}(y) - g_{Y}(y)] dy \\
&= \int \mathds{1} (y \leq y_{0}) [ f_{X}(b(y))b^{\prime}(y) - g_{X}(b(y))b^{\prime}(y)  ] dy \\
&= \int \mathds{1} (x \leq b(y_{0})) [ f_{X}(b(y))b^{\prime}(y) - g_{X}(b(y))b^{\prime}(y)  ] dy \\
&= \int \mathds{1} (x \leq x_{0}) [ f_{X}(x) - g_{X}(x) ] dx.
\end{align*}

Therefore,
\begin{equation*}
\sup_{y_{0}} \left| F_{Y}(y_{0}) - G_{Y}(y_{0}) \right| = \sup_{x_{0}} \left| F_{X}(x_{0}) - G_{X}(x_{0}) \right|,
\end{equation*}
and the Kolmogorov-Smirnov distance is invariant under one-to-one transformations.
\qed
\end{proof}

\begin{proposition}
\label{propostion9}
The Kolmogorov-Smirnov distance is discretization robust.
\end{proposition}

\begin{proof}
\label{proof_of_proposition9}
\smartqed
Notice that we can write
\begin{equation*}
\left| F(x_{0}) - G(x_{0}) \right| = \left| \int \mathds{1} (x \leq x_{0}) d[F(x) - G(x)] \right|,
\end{equation*}
with $\mathds{1} (x \leq x_{0})$ being thought of as a \textquotedblleft smoothing kernel\textquotedblright. Hence, comparisons between discrete and continuous distributions are allowed and the distance is discretization robust.
\qed
\end{proof}

The Kolmogorov-Smirnov distance is a distance based on the probability integral transform. As such, it is invariant under monotone transformations (see proposition~\ref{proposition8}). A drawback of distances based on probability integral transforms is the fact that there is no obvious extension in the multivariate case. Furthermore, there is not a direct loss function interpretation of these distances when the model used is incorrect. In what follows, we discuss chi-squared and quadratic distances that avoid the issues listed above.


\subsection{Exactly Quadratic Distances}
\label{subsection4.4}

In this section we briefly discuss exactly quadratic distances. Rao \cite{Rao1982} introduced the concept of an exact quadratic distance for discrete population distributions and he called it \textit{quadratic entropy}. Lindsay et al. \cite{Lindsay2008} gave the following definition of an exactly quadratic distance.

\begin{definition}
\label{definition15}
(Lindsay et al. \cite{Lindsay2008}) Let $F$, $G$ be two probability distributions, and $K$ is a nonnegative definite kernel. A quadratic distance between $F$, $G$ has the form
\begin{equation*}
\rho_{K}(F, G) = \iint K_{G}(x, y)d(F-G)(x)d(F-G)(y).
\end{equation*}
\end{definition}

Quadratic distances are of interest for a variety of reasons. These include the fact that the empirical distance $\rho_{K}(\widehat{F}, G)$ has a fairly simple asymptotic distribution theory when $G$ identifies with the true model $\tau$, and that several important distances are exactly quadratic (see, for example, Cram\'er-von Mises and Pearson's chi-squared distances). Furthermore, other distances are asymptotically locally quadratic around $G=\tau$. Quadratic distances can be thought of as extensions of the chi-squared distance class.

We can construct an exactly quadratic distance as follows. Let $F$, $G$ be two probability measures that a random variable $X$ may follow. Let $\varepsilon$ be an independent error variable with known density $k_{h}(\varepsilon)$, where $h$ is a parameter. Then, the random variable $Y=X+\varepsilon$ has an absolutely continuous distribution such that
\begin{equation*}
f_{h}^{*}(y) = \int k_{h}(y-x)dF(x), \quad \text{if} \ X \sim F,
\end{equation*}
or
\begin{equation*}
g_{h}^{*}(y) = \int k_{h}(y-x)dG(x), \quad \text{if} \ X \sim G.
\end{equation*}
Let
\begin{equation*}
P^{*2} (F, G) = \int \frac{[ f^{*}(y) - g^{*}(y) ]^{2}}{ g^{*}(y) } dy,
\end{equation*}
be the kernel-smoothed Pearson's chi-squared statistic. In what follows, we prove that $P^{*2} (F, G)$ is an exactly quadratic distance.

\begin{proposition}
\label{proposition10}
The distance $P^{*2} (F, G)$ is an exactly quadratic distance provided that $\iint K(s, t)d(F-G)(s)d(F-G)(t) < \infty$, where $K(s, t) = \int \frac{k_{h}(y-s) k_{h}(y-t)}{g^{*}(y)} dy $.
\end{proposition}

\begin{proof}
\label{proof_of_proposition10}
\smartqed
Write
\begin{align*}
P^{*2} (F, G) &= \int \frac{[ f^{*}(y) - g^{*}(y) ]^{2}}{ g^{*}(y) } dy \\
&= \int \frac{[ \int k_{h}(y-x)dF(x) - \int k_{h}(y-x)dG(x) ]^{2}}{ g^{*}(y) } dy \\
&= \int \frac{[ \int k_{h}(y-x)d(F-G)(x)]^{2}}{ g^{*}(y) } dy \\
&= \int \frac{[ \int k_{h}(y-s)d(F-G)(s)] [ \int k_{h}(y-t)d(F-G)(t)] }{ g^{*}(y) } dy.
\end{align*}

Now using Fubini's theorem, the above relationship can be written as
\begin{align*}
& \iint \left\{ \int \frac{k_{h}(y-s) k_{h}(y-t)}{g^{*}(y)} dy \right\} d(F-G)(s) d(F-G)(t) \\
&= \iint K(s, t) d(F-G)(s) d(F-G)(t),
\end{align*}
with $K(s, t)$ given above.
\qed
\end{proof}

\begin{remark}
\label{remark7}
(a) The issue with many classical measures of goodness-of-fit is that the balance between sensitivity and statistical noise is fixed. On the other hand, one might wish to have a flexible class of distances that allows for adjusting the sensitivity/noise trade-off. Lindsay \cite{Lindsay1994} and Basu and Lindsay \cite{BasuLindsay1994} introduced the idea of smoothing and investigated numerically the blended weighted Hellinger distance, defined as
\begin{equation*}
BWHD_{\alpha} (\tau^{*}, m_{\theta}^{*}) = \int \frac{ ( \tau^{*}(x) - m_{\theta}^{*}(x) )^{2} }{ \left(\alpha \sqrt{\tau^{*}(x)} + \overline{\alpha} \sqrt{ m_{\theta}^{*}(x) } \right)^{2} } dx,
\end{equation*}
where $\overline{\alpha}= 1- \alpha$, $\alpha \in [1/3, 1]$. When $\alpha=1/2$, the $BWHD_{1/2}$ equals the Hellinger distance.

(b) Distances based on kernel smoothing are natural extensions of the discrete distances. These distances are not invariant under one-to-one transformations, but they can be easily generalized to higher dimensions. Furthermore, numerical integration is required for the practical implementation and use of these distances.
\end{remark}


\section{Discussion}
\label{secction5}

In this paper we study statistical distances with a special emphasis on the chi-squared distance measures. We also introduce an extension of the chi-squared distance, the quadratic distance, introduced by Lindsay et al. \cite{Lindsay2008}. We offered statistical interpretations of these distances and showed how they can be obtained as solutions of certain optimization problems. Of particular interest are distances with statistically interpretable constraints such as the class of chi-squared distances. These allow the construction of confidence intervals for models. We further discussed robustness properties of these distances, including discretization robustness, a property that allows discrete and continuous distributions to be arbitrarily close. Lindsay et al. \cite{Lindsay2014}  study the use of quadratic distances in problems of goodness-of-fit with particular focus on creating tools for studying the power of distance-based tests. Lindsay et al. \cite{Lindsay2014} discuss one-sample testing and connect their methodology with the problem of kernel selection and the requirements that are appropriate in order to select optimal kernels. Here, we outlined the foundations that led to the aforementioned work and showed how these elucidate the performance of statistical distances as inferential functions.


\begin{acknowledgement}
The first author dedicates this paper to the memory of Professor Bruce G. Lindsay, a long time collaborator and friend, with much respect and appreciation for his mentoring and friendship. She also acknowledges the Department of Biostatistics, School of Public Health and Health Professions and the Jacobs School of Medicine and Biomedical Sciences, University at Buffalo, for supporting this work. The second author acknowledges the Troup Fund, Kaleida Foundation for supporting this work.
\end{acknowledgement}


%

\end{document}